\documentclass{article}

\usepackage{amsthm}
\usepackage{amsmath}
\usepackage{amssymb}

\usepackage{enumitem}

\newtheorem{theorem}{Theorem}[section]

\newtheorem{lemma}[theorem]{Lemma}
\newtheorem{claim}[theorem]{Claim}

\theoremstyle{definition}
\newtheorem{definition}[theorem]{Definition}

\newtheorem{observation}[theorem]{Observation}

\newcommand{\R}{\mathbb R}
\newcommand{\Z}{\mathbb Z}

\title{Is the space of reachable\\ particle configurations dense?}
\author{J\'anos Pach \thanks{R\'enyi Institute, Budapest. Research partially supported by ERC Advanced Grant ``GeoScape'' and by the IAS/Park City Mathematics Institute during the summer program of 2025.  Email: pach@renyi.hu.} \and G\'abor Tardos \thanks{R\'enyi Institute, Budapest. Research partially supported by National Research, Development and Innovation Office (NKFIH) grants K-132696, SSN-135643, and ERC Advanced Grant ``ERMiD.'' Email: tardos@renyi.hu. }}
\date{}
\begin{document}
\maketitle
\begin{abstract}
Let $p_0,\ldots,p_n$ be a finite sequence of points in an Euclidean space $\R^d$. Suppose that there is a (pointlike) particle sitting at each point $p_i$. In a ``legal'' move, any one of them can jump over another, landing on the other side, at exactly the same distance. Under what circumstances can we guarantee that for any $\varepsilon>0$ and any other sequence of points $q_0,\ldots, q_n\in\R^d$, there is a finite sequence of legal moves that takes the particle at $p_i$ to the $\varepsilon$-neighborhood of $q_i$, simultaneously for every $i$?

We prove that this is possible if and only if the additive group generated by the vectors $p_1-p_0,\ldots,p_n-p_0$ is dense in $\R^d$.
\end{abstract}

\section{Introduction}


Let $n$ and $d$ be a positive integers. Suppose that we have $n+1$ labeled pointlike particles, originally sitting at the points $p_0, p_1, \ldots, p_n\in\R^d$.
At any given moment, one of the particles (the $i$th one, say), currently sitting at a point $p_i'$, will jump over another one (the $j$th one, say), currently at the location $p_j'$, and land at the point symmetric about it, that is, at the point
$$p_i'+2(p_j'-p_i')=2p_j'-p_i'.$$
We call such a move \emph{legal}. It is permitted that sometimes two or more different particles sit at the same place.

The set of configurations $(p'_0, p'_1, \ldots, p'_n)$ that can be reached from a given initial configuration $(p_0, p_1, \ldots, p_n)$ by legal moves, is a subset of $\R^d\times\cdots\times\R^d=\R^{d(n+1)}$. It is called the \emph{space of reachable configurations}.
\smallskip

The aim of the present note is to investigate that under what conditions the space of reachable configurations is \emph{dense} in $\R^{d(n+1)}$.
\medskip

\begin{definition} \emph{An initial particle configuration $(p_0, p_1, \ldots, p_n)$ is said to be \emph{rich} if the space of configurations reachable from it is \emph{dense} in $\R^{d(n+1)}$.}
\end{definition}

In other words, $(p_0, p_1, \ldots, p_n)$ is rich if for every $q_0, q_1,\ldots, q_n\in\R^d$ and for every $\varepsilon>0,$ there is a sequence of legal moves that take the $i$\/th particle (starting at $p_i$) to a position at distance at most $\varepsilon$ from $q_i,$ simultaneously for all $0\le i\le n$. We use the Euclidean distance and norm in $\R^d$, but we note that the choice of norm is irrelevant.

Clearly, $(p_0, p_1, \ldots, p_n)$ is \emph{rich} if and only if $(0,p_1-p_0,\ldots, p_n-p_0)$ is rich. Therefore, in order to characterize all rich configurations, it is enough to consider the cases where $p_0=0$.

Throughout this note, we will assume that $p_0=0$ and we will identify the initial configuration $(0,p_1,\ldots,p_n)$ with the $d\times n$ matrix
$$P=(p_1\;\;p_2\;\;\dots\;\;p_n).$$
We write $G(P)$ for the \emph{additive subgroup of $R^d$ generated} by the vectors $p_1,\ldots,p_n$. In other words, we have $G(P)=\{Pv\mid v\in\Z^d\}$.

\smallskip

The aim of this note is to establish the following result.

\begin{theorem}\label2
Let $n, d>0$ be integers, and let $P=(0,p_1,\ldots, p_n)$ be a configuration of $n+1$ points in $\R^d$. The following two conditions are equivalent.
\begin{enumerate}
    \item $P$ is a rich configuration;
    \item $G(P)$ is \emph{dense} in $\R^d$.
\end{enumerate}
\end{theorem}

It is trivial that condition 1 implies condition 2. Indeed, legal moves keep all particles in $G(P)$, so by legal moves we cannot bring the first particle close to every prescribed point in $\R^d$, unless $G(P)$ is dense in $\R^d$. Much less can we bring all the particles, simultaneously, close to every prescribed configuration.

Proving that condition 2 implies condition 1 is more involved. The bulk of this note is devoted to this task.

Before turning to the proof, we make some simple observations.

\begin{observation}\label{obs1}
\emph{If $G(P)$ is dense in $\R^d$, then $n>d$.}
\end{observation}

\begin{proof}
Note that $G(P)$ is contained in the subspace generated by the columns of $P$ over the reals, so if $G(P)$ is dense in $\R^d$, then this subspace must be the entire space. This implies that $n\ge d$. If $n=p$, the columns of $P$ may linearly generate the entire space, but in this case $G(P)$ is a discrete lattice, i.e., it is not dense in $\R^d$. Thus, $n$ must be strictly greater than $d$.
\end{proof}

There is a simple reformulation of the condition that $G(P)$ is not dense to a positive statement which may be easier to handle, see Lemma~\ref3 below. As we were unable to find a reference for it, we include a proof in the Appendix.
\smallskip

We write $M^T$ to denote the \emph{transpose} of a matrix $M$. The vectors $w\in\R^d$ are considered column vectors, so the corresponding row vector is $w^T$.

\begin{lemma}\label3
Let $n, d>0$ be integers, and let $P$ be a $d\times n$ real matrix. The following two conditions are equivalent.
\begin{enumerate}
\item $G(P)$ is not dense in $\R^d$;
\item There exists a nonzero vector $w\in\R^d$ such that $w^TP$ is an integer vector.
\end{enumerate}
\end{lemma}

However, for the proof of Theorem~\ref2, we need another equivalent form of the property that $G(P)$ is dense. To formulate our statement, we need to introduce a special class of matrices.

\begin{definition} \emph{A \emph{step matrix} is a square matrix that is obtained from the identity matrix by replacing one of its off-diagonal entries by $2$ or $-2$.}

\emph{A matrix is called \emph{good} if it belongs to the multiplicative group generated by step matrices.}
\end{definition}

The main technical tool required for the proof of Theorem~\ref2 is the following.

\begin{theorem}\label{fo}
Let $n, d>0$ be integers, and let $P$ be a $d\times n$ real matrix. The following two statements are equivalent.
\begin{enumerate}
\item $G(P)=\{Pv\mid v\in\Z^n\}$ is dense in $\R^d$;
\item $\{PA\mid A \hbox{ is good}\}$ is dense among the $d\times n$ real matrices.
\end{enumerate}
\end{theorem}

The problem of analyzing the movement of particles according to the above rules was originally suggested by Alexander Kovaldzhi~\cite{F06}. Answering a question of Florestan Brunck~\cite{B}, we gave a simple characterization of all configurations of particles that can be reached from a given configuration, using only legal moves. This enabled us to show that if the initial particle configuration is a regular $(n+1)$-gon in the plane, then the space of reachable configurations always contains a strictly larger regular $(n+1)$-gon, unless $n+1=3$, $4$, or $6$; see~\cite{PaT24}.
\medskip

In Section~\ref{sec3}, we establish some elementary properties of step matrices, and give a characterization of good matrices. These properties will be used in Section~\ref{sec4} to prove Theorem~\ref{fo}. In Section~\ref{sec5}, we prove Theorem~\ref2 by combining Theorem~\ref{fo} and the characterization of reachable configurations, given in \cite{PaT24}. The last section contains some concluding remarks and open problems.
\medskip

The problem addressed in this note is broadly related to several topics in pure and applied mathematics and computer science. For instance, the theory of \emph{discrete dynamical systems} deals with models in which changes occur in discrete time steps, according to a fixed rule; see~\cite{Ga}. Solitary games in which the player can choose one from several possible legal moves, are studied in the framework of \emph{combinatorial game theory}~\cite{AlNW, BeCG}. Similar models are analyzed in \emph{motion planning} and \emph{robotics}, related to the reconfiguration of rotating or sliding systems of modules~\cite{DuSY}.

\section{Elementary properties of step matrices}\label{sec3}

Recall that a step matrix is a square matrix which is obtained from the identity matrix by writing $2$ or $-2$ instead of one of its off-diagonal entries (which was originally $0$). It is easy to check that the $i$\/th power of a step matrix can be obtained from it by multiplying its $\pm2$ entry by $i$ and in particular, the inverse of a step matrix is also a step matrix.

Notice that all step matrices have determinant $1$, and thus the good matrices (the matrix group generated by the step matrices) form a subgroup of $SL(n,\Z)$. Some elementary properties of good matrices, needed in the rest of this note, are summarized in the following lemma.

\begin{lemma}\label1
\begin{itemize}
\item[(i)]
Let $n$ be a positive integer and $v\in\R^n$. There exists a $n\times n$ step matrix $A$ such that the vector $Av$ is shorter than $v$, unless every entry of $v$ is equal to $0$, $a$, or $-a$, for some $a\in\R$.
\item[(ii)]
Let $n$ be a positive integer and $v\in\Z^n$ a nonzero vector. There exists a good $n\times n$ matrix $A$ such that the only entries of $Av$ are $a$, $-a$, and $0$, where $a$ is the greatest common divisor of the entries of $v$.
\item[(iii)]
A square matrix $A$ is good if and only if
\begin{enumerate}
\item $\det(A)=1$;
\item the non-diagonal entries of $A$ are even integers; and
\item the diagonal entries of $A$ are integers congruent to $1$ modulo $4$.
\end{enumerate}
\end{itemize}
\end{lemma}

\begin{proof}
To see part~(i), consider a vector $v\in\R^n$ with two entries $a$ and $b$ satisfying $|b|>|a|>0$. If we multiply $v$ by a suitable step matrix $A$ we can change its entry $b$ to either $b+2a$ or $b-2a$, without altering any other entries. One of these changes decreases the Euclidean norm of the vector, that is, we have $||Av||<||v||.$
\medskip

For part~(ii), we take a vector $v\in\Z^n$ and apply part~(i) to it, repeatedly. The length of the resulting integer vector decreases in every step. Hence, we must get stuck after finitely many steps at an integer vector $Av$, where $A$ is good and all nonzero entries of $Av$ have the same absolute value. During this process, the greatest common divisor, $a$, of the entries of the vector does not change, therefore all entries of $Av$ must be $a$, $-a$ or $0$, as claimed.
\medskip

Notice that every step matrix satisfies all three conditions listed in part~(iii). Moreover, it is easy to see that the set of matrices $A$ satisfying these conditions form a subgroup. This proves the ``only if'' direction of part~(iii).
\medskip

It remains to prove the ``if'' direction of part~(iii). We proceed by induction on $n$. The case $n=1$ is trivial (only the identity matrix satisfies the conditions), so we assume $n>1$ and also that the statement holds for smaller matrices.

\smallskip

Consider an $n\times n$ integer matrix $X$ satisfying the three conditions in the statement. We will multiply $X$ from the left with good matrices. Note that the matrices we obtain during this process must all satisfy our three conditions.
Our goal is to transform the matrix into simpler and simpler matrices. When we arrive at the identity matrix, we have succeeded in proving that $X$ is good.

First, we apply part~(ii) of the lemma to the last column $v$ of $X$. Note that the greatest common divisor of the entries of $v$ divides $\det(X)=1$. Thus, we can find a good matrix $A$ with all entries of $Av$ being $1$, $-1$ or $0$. Note that the last column of $AX$ is $Av$, and $AX$ satisfies the divisibility conditions, so its last column, $Av$, must be equal to $e_n$, the last standard basis vector of $\R^n$ (i.e., all-zero with a single $1$ entry in the last place).

Let $B$ denote the submatrix of $AX$ formed by the first $n-1$ rows and columns. It must satisfy the same divisibility conditions as $AX$ does. Expanding the determinant of $AX$ along its last column, we obtain that $\det(B)=\det(AX)=1.$ By the induction hypothesis, $B$ must be good. So, in particular, $B^{-1}$ can be written as product of step matrices: $B^{-1}=\prod_{i=1}^kC_i$. Let $f(C)$ stand for the extension of the $(n-1)\times(n-1)$ matrix $C$ by a last row and column containing only zeros except for the $1$ in the diagonal. The function $f$ preserves multiplication, so we have $f(B^{-1})=\prod_{i=1}^kf(C_i)$. Here, $f(C_i)$ is a step matrix for every $i$, so $f(B^{-1})$ is also a good matrix.

Let us consider now the matrix $f(B^{-1})AX$. It differs from the identity matrix only in the first $n-1$ entries in the last row. These entries must be even, as our divisibility conditions must be satisfied. The $i$\/th entry in the last row can then be get rid of by multiplying it with the appropriate power of the step matrix having $2$ in the same position. This completes our project to obtain the identity matrix by multiplying $X$ with good matrices and, therefore, shows that $X$ is good.
\end{proof}

Given a not necessarily square $n\times d$ matrix $X=(x_{ij})^{1\le i\le n}_{1\le j\le d}$, its \emph{main diagonal} or, in short, \emph{diagonal} consists of all entries of the form $x_{ii}$. Accordingly, the entries \emph{below the diagonal} are $x_{ij}$ with $i>j$. Note that if $n>d$, then all the entries in the last $n-d$ rows of $A$ are below the diagonal. We will need the following statement that is somewhat similar to Lemma~\ref1(ii).

\begin{lemma}\label{gen}
Let $n$ and $d$ be positive integers, $n\ge d$, and let $X$ be an $n\times d$ integer matrix with all of its diagonal entries odd and all of its entries below the diagonal even.

Then there exists a good $n\times n$ matrix $A$ such that all entries of $AX$ below the diagonal are zero.
\end{lemma}

\begin{proof}
We proceed by induction on $d$. For $d=1$, $X$ is a column vector, so we can apply Lemma~\ref1(ii) to obtain a good matrix $A$ with all nonzero entries of $AX$ having the same absolute value. It follows from the parity conditions in Lemma~\ref1(iii) that the first entry of $AX$ is odd, and all other entries are even. This implies that the even entries have to be $0$, and we are done.
\smallskip

Consider now the case $d>1$, assuming that we have already proved the lemma for matrices with fewer than $d$ columns. First, we apply the case $d=1$ of the lemma to the first column $Xf_1$ of the matrix $X$. (Here $f_1$ denotes the first standard basis vector of $\R^d$.) We obtain a good matrix $A$ such that all entries of $AXf_1$, the first column of $AX$, are zero, except the first entry.

Let $Y$ stand for the matrix obtained from $AX$ by removing its first column and first row. Clearly, $Y$ is an $(n-1)\times(d-1)$ matrix that also satisfies the parity assumption in the lemma, as $Y$ and $AX$ agree modulo $2$. Applying the inductive hypothesis to $Y$, we find a good $(n-1)\times(n-1)$ matrix $B'$ such that all entries below the diagonal in $B'Y$ are zero. Extend $B'$ with a new first row and column containing only zero entries, except a $1$ at its top left position. By Lemma~\ref1(iii), the resulting matrix $B$ is also a good matrix. We finish the proof by observing that $BA$ is good and all entries of $BAX$ below the diagonal are zero.
\end{proof}

\section{Proof of Theorem~\ref{fo}}\label{sec4}

Condition 2 readily implies condition 1. Indeed, if the set of matrices $\{PA\mid A \hbox{ is good}\}$ is dense in the set of all $d\times n$ real matrices, then the set of their first columns (which are of the form $Pv$ for an integer vector $v$, namely, the first column of a good matrix $A$) is dense in $\R^d$.
\medskip

To prove that condition 1 implies condition 2,
we use induction on $d$. The basis of the induction is the degenerate case $d=0$. Clearly, the ``empty matrix'' (of size $0\times n$) satisfies both conditions in the theorem. Assume now that $d\ge1$, and that the implication holds for smaller values of $d$.

Let us fix a $d\times n$ real matrix $P$ that satisfies condition~1. It follows from Observation~\ref{obs1} that $n>d$. We call an integer vector \emph{primitive} if the greatest common divisor of its coordinates is $1$.

\begin{claim}\label{claim1}
There exists a primitive vector $v\in\Z^n$ such that $Pv$ nonzero, but arbitrarily short and all coordinates of $v$ are even, except for its last coordinate which is 1 modulo 4.
\end{claim}

\begin{proof}
    By our assumptions on $P$, we can find a vector $v_0\in\Z^n$ such that $Pv_0$ is distinct from, but arbitrarily close to $Pe_n/2$, where $e_n$ is the last standard basis vector in $\R^n$. Set $v_1=2v_0-e_n$. Clearly, $Pv_1$ is nonzero, but arbitrarily short, while all coordinates of $v_1$ are even except for the last one, which is odd. Let $k$ be the greatest common divisor of the coordinates of $v_1$, this is an odd integer. Clearly, $v_2=v_1/k$, is a primitive vector with all its coordinates even except the last one that is odd. We have $Pv_2=Pv_1/k$, so it is nonzero and can be made arbitrarily short. One of $v=v_2$ or $v=-v_2$ satisfies all the requirements and establishes the claim.
\end{proof}

Applying Lemma~\ref1(ii) to $v$, we immediately obtain that there exists a good $n\times n$ matrix $A$ such that all entries of $Av$ are equal to $\pm1$ or $0$. It follows from the parity and modulo 4 constraints that the only possibility is that $Av$ is equal to $e_n,$ the last standard basis vector in $\R^n$ (cf. Lemma~\ref1(iii)). Thus, we have
\begin{claim}\label{claim2}
There exists a good $n\times n$ matrix $A$ such that $Av=e_n$
\end{claim}

Choose an orthogonal transformation $O$ of $\R^d$ that moves $Pv$ to a vector parallel to $f_d$, the last standard basis vector in $\R^d$. So, we have $$OPv=\varepsilon f_d,$$ where, according to Claim~\ref{claim1}(i), $$|\varepsilon|=||OPv||=||Pv||$$ can be made arbitrarily small, while keeping it nonzero.
\smallskip

Consider now the matrix $D=OPA^{-1}$. The last column of $D$ is $$De_n=OPA^{-1}e_n=OPv=\varepsilon f_d.$$ Let $P'$ be the the matrix formed by the first $d-1$ rows and the first $n-1$ columns of $D$, and let $w^T$ stand for the first $d-1$ entries of the last row of $D$. We have
\begin{equation}\label{eq1}
D=\begin{pmatrix}P'&0\\w^T&\varepsilon\end{pmatrix}.
\end{equation}

We claim that $P'$ satisfies condition 1 in Theorem~\ref{fo}.

\begin{claim}\label{claim3}
The set of vectors $\{P'u'\mid u'\in\Z^{n-1}\}$ is dense in $\R^{d-1}$.
\end{claim}

\begin{proof}
    We have $\{PA^{-1}u\mid u\in\Z^d\}=\{Pu\mid u\in\Z^d\},$ because $A$ is in $SL(n,\Z)$. By assumption, the letter set is dense in $\R^d$. Therefore, the image of this set under an orthogonal transformation, $$\{OPA^{-1}u\mid u\in\Z^d\}=\{Du\mid u\in\Z^d\},$$ is also dense. The first $d-1$ coordinates of $Du$ are given by $P'u'$, where $u'$ stands for the vector formed by the first $n-1$ coordinates of $u$. Thus, $\{P'u'\mid u'\in\Z^{n-1}\}$ has to be dense in $\R^{d-1}$,
\end{proof}

From Claim~\ref{claim3} and the induction hypothesis, we conclude that $P'$ satisfies condition~2 of Theorem~\ref{fo}. In other words, we have the following.

\begin{claim}\label{claim4}
The set of matrices $\{P'B'\mid B' \hbox{ is good}\}$ is dense among the $(d-1)\times(n-1)$ real matrices.
\end{claim}

Given a \emph{good} $(n-1)\times(n-1)$ matrix $B'$ and a vector $y\in\Z^{n-1}$ whose every coordinate is \emph{even}, let us construct an $n\times n$ matrix $B$, as follows.
\begin{equation}\label{eq2}
B=B(B',y)=\begin{pmatrix}B'&0\\y^T&1\end{pmatrix}.
\end{equation}
Using the characterization of good matrices given in Lemma~\ref1(iii), we obtain
\begin{claim}\label{claim5}
The matrix $B$ defined in~{(\ref{eq2})} is a good matrix. Hence, $A^{-1}B$, the product of two good matrices, is also good.
\end{claim}

From (\ref{eq1}) and (\ref{eq2}), we get
\begin{equation}\label{eq3}
DB=\begin{pmatrix}P'B'&0\\w^T B'+\varepsilon y^T&\varepsilon\end{pmatrix}.
\end{equation}

Consider now a $d\times n$ real matrix $X$ whose last column is $0$. Then the last column of $OX$ is also $0$. By Claim~\ref{claim4}, selecting a suitable $(n-1)\times(n-1)$ good matrix $B'$, we can attain that the first $d-1$ rows of $DB$ are arbitrarily close to the first $d-1$ rows of $OX$, where $B=B(B',y)$ and $y$ is arbitrary. Then we can select a suitable even vector $y\in\Z^{n-1}$ to achieve that every coordinate of the last row of $DB$ is $|\varepsilon|$-close to the last row of $OX$. Summarizing:

\begin{claim}\label{claim6}
For every $\delta>|\varepsilon|$, there exists a good $n\times n$ matrix $B$ such that every column of $DB=OPA^{-1}B$  is $\delta$-close to the corresponding column of $OX$.

Therefore, every column of $PA^{-1}B$ is $\delta$-close to the corresponding column of $X$. Note that $|\varepsilon|$ and hence $\delta$where $\delta$ can be made arbitrarily small.
\end{claim}

The second statement of the claim directly follows from the first one, because $O$ is an orthogonal transformation of $\R^d$.
\smallskip

To prove Theorem~\ref{fo}, we need to show that for a fixed $d\times n$ real matrix $P$ which satisfies condition 1, the set of matrices $\{PC\mid C \hbox{ is good}\}$ is dense among all $d\times n$ real matrices.
In Claim~\ref{claim6}, we have shown that every $d\times n$ real matrix $X$ whose last column is zero can be approximated arbitrarily closely by matrices of the form $PC$, where $C$ is good. In the rest of the proof, we will get rid of the assumption that the last column of $X$ must be zero.
\smallskip

From now on, let $X$ denote an arbitrary $d\times n$ real matrix $X$, and let $\varepsilon>0$.

First, we approximate $(1/\varepsilon)X^T$ by an $n\times d$ integer matrix $Y$ with odd entries in the main diagonal and even entries in the non-diagonal positions. By Observation~\ref{obs1}, we have that $n>d$, so all entries of the last $n-d$ rows of $(1/\varepsilon)X^T$ are below the diagonal. Clearly, the approximating matrix $Y$ can be chosen in such a way that every entry of it is within $1$ of the matrix $(1/\varepsilon)X^T$.
By Lemma~\ref{gen}, there exists a good $n\times n$ matrix $A_0$ such that all entries of $A_0Y$ below the diagonal are zero, so, in particular, the last row is zero.
\smallskip

Consider now the $d\times n$ matrix $X_0=\varepsilon Y^TA_0^T$. Its last column is zero. Therefore, our theorem is true for $X_0$, which means that $X_0$ can be arbitrarily closely approximated by a matrix of the form $PB$, where $B$ is good. Equivalently, $PB(A_0^T)^{-1}$ can approximate $X_0(A_0^T)^{-1}=\varepsilon Y^T$ arbitrarily well. Here $A_0$ and $B$, hence also $C=B(A_0^T)^{-1}$ are good matrices.

Each entry of $\varepsilon Y^T$ is $\varepsilon$-close to the corresponding entry of $X$. Since $\varepsilon>0$ was arbitrary, we can conclude that $X$ can be approximated arbitrarily closely by matrices of the form $PC$, where $C$ is good. This completes the induction and, hence, the proof of Theorem~\ref{fo}.

\section{Proof of Theorem~\ref2}\label{sec5}

In an earlier paper~\cite{PaT24}, we gave a linear algebraic description of all configurations $(0,q_1,\ldots,q_n)$ in $\R^d$, that can be reached from the initial configuration $(0,p_1,\ldots,p_n)$ by legal moves, under the assumption that particle $0$ at the origin never moves. We call this particle \emph{stationary}.

As before, we identify every such configuration with the $d\times n$ matrix whose $i$\/th column consists of the coordinates of the $i$\/th particle, $1\le i\le n.$ To describe the set of reachable configurations, in~\cite{PaT24} we introduced some simple $n\times n$ matrices, called \emph{elementary matrices} (or \emph{elementary involutions}).
\smallskip

There are two types of elementary matrices. An elementary matrix of the \emph{first type} is a matrix $A_i$ that can be obtained from the identity matrix by replacing its $i$\/th diagonal entry by $-1$, for $1\le i\le n.$ $A_i$ corresponds to moving particle $i$ over the stationary particle sitting at the origin. If $P$ is the matrix of the original configuration, then after this move, we obtain a configuration whose matrix is $PA_i$.

An elementary matrix of the \emph{second type}, is a matrix $A_{ij}\;(1\le i\neq j \le n)$ that can be obtained from $A_i$ by replacing the $j$\/th entry of its $i$\/th column by $2$. The matrix $A_{ij}$ corresponds to moving particle $i$ over particle $j$. As before, we get the matrix of the resulting configuration from that of the original one by multiplying it with $A_{ij}$.

For example, for $n=4$, we have
\begin{equation}\label{matrix}
A_{2}=\begin{bmatrix}
1&0&0&0\\
0&-1&0&0\\
0&0&1&0\\
0&0&0&1
\end{bmatrix}\quad  \mbox{and} \quad
A_{24}=\begin{bmatrix}
1&0&0&0\\
0&-1&0&0\\
0&0&1&0\\
0&2&0&1
\end{bmatrix},
\end{equation}
where the first matrix corresponds to the jump of particle $2$ over the stationary particle $0$, and the second matrix to the jump of $2$ over the particle $4$.

\begin{definition}\label{def3}
\emph{An $n\times n$ matrix is called \emph{admissible} if it can be obtained as a product of elementary matrices.}
\end{definition}

Since every elementary matrix is its own inverse, all admissible matrices are invertible.
\smallskip

In~\cite{PaT24}, we established the following lemma.

\begin{lemma}\label{altalanos}
Let $P$ be a $d\times n$ matrix associated with the initial configuration of $n+1$ particles in ${\mathbf{R}}^d$, one of which is sitting at the origin $0$.

(i) The set of configurations that can be reached from $P$ by a sequence of legal moves that keep the stationary particle at the origin, is $\{PA\mid A \hbox{ is admissible}\}$.

(ii) If we drop the assumption that the stationary particle must stay at the origin and allow it to jump over any other particle, then a configuration can be reached from $P$ if and only if it can be obtained from a configuration that can be reached without moving the $0$\/th particle by translating it by a vector of the form $2Pw$, for some $w\in\Z^n.$
\end{lemma}

To complete the proof of Theorem~\ref2, we need another simple
\begin{lemma}\label{obs2}
Every good matrix is admissible.
\end{lemma}

\begin{proof}
Since every good matrix is a product of step matrices, it is enough to verify that every step matrix can be written as a product of elementary matrices. We illustrate the easy proof by an example. The $4\times 4$ step matrix $M$ obtained from the identity matrix by replacing the 2nd entry of its last row by $-2$ can be written as $M=A_{24}A_2$, where $A_{24}$ and $A_2$ are the elementary matrices shown in (\ref{matrix}). Hence, for the matrix $M'$ obtained from the identity matrix by replacing the 2nd entry of its last row by $+2$, we have $M'=M^{-1}=A_2^{-1}A_{24}^{-1}=A_2A_{24}$ 
\end{proof}

\begin{proof}[Proof of Theorem~\ref2] As we pointed out in the Introduction, condition 1 of the theorem trivially implies condition 2. In every configuration that can be reached from $P$, each particle remains in $G(P)$. Thus, if any configuration can be approximated arbitrarily closely by configurations reachable from $P$, then $G(P)$ must be dense in $\R^d$.
\medskip

Next we prove that condition 2 implies condition 1. Let us fix a configuration $(0,p_1,\dots,p_n),$ identified with the $d\times n$ matrix $P$ having $p_i$ as its $i$\/th column. Suppose that $G(P)$ is dense in $\R^d$. We have to prove that any configuration $(q_0,q_1,\dots,q_n)$ of $n+1$ points in $\R^d$ can be approximated arbitrarily closely by configurations reachable from $P$.
\smallskip

Let us fix such a configuration $(q_0,q_1,\dots,q_n),$ and  let $\varepsilon>0$ be an arbitrarily small constant. According to our assumption (condition 2),
$$G(P)=\{Pw\mid w\in\Z^d\}$$ is dense in $\R^d$, so we can find a vector $w_0\in \Z^d$ such that $Pw_0$ is $\varepsilon/2$-close to $q_0/2$.

By Theorem~\ref{fo}, $\{PA\mid A \hbox{ is good}\}$ is dense among all $d\times n$ real matrices. Hence, we can choose a good matrix $A_0$ such that $PA_0$ is $\varepsilon$-close to the matrix $Q$, whose $i$\/th column is $q_i-q_0$, for $1\le i\le n.$ By Lemma~\ref{obs2}, $A_0$ is not only good, but also admissible.

It follows from Lemma~\ref{altalanos}(ii) that the configuration associated with the matrix $PA_0$, translated by the vector $2Pw_0$, can be reached from $P$ by a sequence of legal moves. Moreover, it is $\varepsilon+\varepsilon=2\varepsilon$-close to the configuration $(q_0,q_1,\ldots,q_n).$ \end{proof}

\section{Concluding remarks}\label{sec6}

1. According to Lemma~\ref{obs2}, the group generated by the $n\times n$ step matrices is a subgroup of the group generated by the elementary matrices (involutions) of the same size. It is easy to see that the index of this subgroup is $2^n$. The characterization of this subgroup, given in part (iii) of Lemma~\ref1 is similar to the characterization of admissible matrices given in \cite{PaT24}.
\smallskip

\noindent 2. Our approach heavily uses the fact that reversing a legal move we obtain another legal move. Therefore, the matrices describing these moves are invertible, with the inverses also describing legal moves. However, one can no longer explore the group structure if we slightly change the definition of a legal move. For instance, we can call a move legal if particle $p_i$ jumping over particle $p_j$, lands on the other side of $p_j$ at a distance $c||p_i-p_j||$, where $c>0$ is a constant, different from 1.

\medskip

\appendix\section*{Appendix}
\section*{Proof of Lemma~\ref3}

It is easy to see that condition~2 of the lemma implies condition~1. Indeed, suppose $w\in\R^d$ is a nonzero vector such that $w^TP$ is an integer vector. Then we have $w^TPv\in\Z,$ for every $v\in\Z^n$. Thus, $$G(P)\subseteq\{x\in\R^d\mid w^Tx\in\Z\},$$ but even the latter set is not dense in $\R^d$.
\medskip

To see the reverse implication, suppose that $G(P)$ is \emph{not dense} in $\R^d$, that is, the \emph{topological closure} $\overline{G}$ of $G(P)$ is not equal to $\R^d$. We claim that then there exists a nonzero vector $w\in\R^d$ such that $w^TP$ is an integer vector.

We prove this claim by induction on $d$. In the degenerate case $d=0$ $\overline(G)=G(P)=\R^0$, so there is nothing to prove. Assume that $d\ge1$ and that we have already proved the claim for dimension $d-1$. Notice that $\overline{G}$ is also an additive subgroup of $\R^d$, just like $G(P)$ is.

We distinguish two cases, according to the lengths of the nonzero vectors in $\overline{G}$.
\medskip

\noindent\textbf{Case A:} \emph{There is $\delta>0$ such that $||z||\ge\delta$, for every nonzero $z\in \overline{G}$.}

Then we have $\overline{G}=G(P)$, and $G(P)$ is a \emph{discrete} set, i.e., the minimum distance between its elements is at least $\delta$. It is well known~\cite{Gr07} (Theorem 21.2) that in this case $G(P)$ is a (possibly lower dimensional) \emph{lattice} generated by at most $d$ linearly independent vectors $u_1, u_2,\ldots, u_k\in G(P)$. We can choose a vector $w\in\R^d$ such that $w^Tu_i=1$, for every $i\; (1\le i\le k)$. This implies that $w\ne0$, but $w^Tz$ is an integer for every $z\in G(P)$. In particular, $w^TP$ is an integer vector, as required.
\medskip

\noindent\textbf{Case B:} $\inf\{||z|| : 0\neq z\in \overline{G}\}=0$.

Choose an infinite sequence of nonzero vectors $u_1, u_2,\ldots,$ from $\overline{G}$ tending to zero such that their directions also converge, say, to the direction of a \emph{unit} vector $u\in \R^d$.  For any real number $\alpha$, we have
$$\lim_{j\rightarrow\infty}\left\lfloor\frac{\alpha}{||u_j||}\right\rfloor u_j=\frac{\alpha}{||u||} u= \alpha u.$$
Since every vector on the left-hand side belongs to $\overline{G}$, so does $\alpha u$. Then the whole line $L$ induced by the vector $u$ belongs to $\overline{G}$. Let $L^{\perp}$ denote the $(d-1)$-dimensional subspace of $\R^d$ orthogonal to $L$, and let $G'=\overline{G}\cap L^{\perp}.$ Clearly, $G'$ is a closed additive subgroup of $L^{\perp}\subset\R^d$, and we have $\overline{G}=G'+L$.

If $G'$ is the whole $(d-1)$-dimensional subspace $L^{\perp}$, then $\overline{G}=G'+L=\R^d$, contradicting our assumption.

If $G'$ is not the whole $(d-1)$-dimensional subspace $L^{\perp}$, then write each $p_i$ in the form $p_i'+ l_i$ for some $p_i\in L^{\perp}$ and $l_i\in L$. Using the induction hypothesis for the vectors $p_i'$ (now considered within the $(d-1)$-dimensional space $L^\perp$), we obtain that there is a nonzero vector $w\in L^{\perp}$ such that $w^Tp'_i$ is an integer for every $i\; (1\le i\le n)$. But clearly, $w^Tp_i=w^Tp'_i$ for every $i$, so $w^TP$ is an integer vector, as needed.

\end{document}